\documentclass{elsarticle}

\usepackage{amssymb}
\usepackage{amsfonts}
\usepackage{verbatim}
\usepackage{xcolor}
\usepackage{amsmath}

\usepackage[a4paper,left=30mm,width=150mm,top=25mm,bottom=25mm]{geometry}

\usepackage{hyperref}

\usepackage{graphicx}
\DeclareGraphicsExtensions{.eps,.ps,.jpg,.bmp}
\usepackage{subfigure}

\newtheorem{theorem}{Theorem}
\newtheorem{lemma}[theorem]{Lemma}

\newdefinition{remark}{Remark}
\newproof{proof}{Proof}

\usepackage{lineno,hyperref}
\modulolinenumbers[5]

\journal{CNSNS}

\biboptions{sort&compress}

\begin{document}

\begin{frontmatter}

\title{Stability for nonlinear wave motions damped by time-dependent frictions\tnoteref{mytitlenote}}
\tnotetext[mytitlenote]{This research was partially supported by the National Natural Science Foundation of China (11802236) and the Fundamental Research Funds for the Central Universities (310201911CX033).}



\author[mymainaddress]{Zhe Jiao}
\ead{zjiao@nwpu.edu.cn}

\author[mymainaddress,mysecondaryaddress]{Yong Xu}
\ead{hsux3@nwpu.edu.cn}

\author[mymainaddress]{Lijing Zhao\corref{mycorrespondingauthor}}
\cortext[mycorrespondingauthor]{Corresponding author}
\ead{zhaolj@nwpu.edu.cn}

\address[mymainaddress]{Department of Applied Mathematics, Northwestern Polytechnical University, Xi'an 710129, People's Republic of China}
\address[mysecondaryaddress]{MIIT key Laboratory of Dynamics and Control of Complex Systems, Northwestern Polytechnical University, Xi'an 710072, People's Republic of China}

\begin{abstract}
We are concerned with the dynamical behavior of solutions to semilinear wave systems with time-varying damping and nonconvex force potential. Our result shows that the dynamical behavior of solution is asymptotically stable without any bifurcation and chaos. And it is a sharp condition on the damping coefficient for the solution to converge to some equilibrium. To illustrate our theoretical results, we provide some numerical simulations for dissipative sine-Gordon equation and dissipative Klein-Gordon equation.
\end{abstract}


\end{frontmatter}


\section{Introduction}
In this paper, we consider the initial-boundary value problem for semilinear wave systems with time-varying dissipation: 
\begin{eqnarray}
\left
    \{
        \begin{array}{ll}
            u_{tt}-\Delta u+ h(t)u_t +f(u)=0 &(t, x)\in\mathbb{R}^{+} \times \Omega,  \\
                                u = 0  & (t, x)\in \mathbb{R}^{+} \times \partial\Omega,\\
                  (u(0, x), u_t(0, x))= (u_{0}(x), u_{1}(x)) & x\in \Omega,
        \end{array} \label{1}
\right.
\end{eqnarray}
where $\Omega$ is a bounded domain in $\mathbb{R}^d$, $d\geq 1$, with smooth boundary $\partial\Omega$, and the initial data  $u_{0} \in H^{2}(\Omega)\cap H^{1}_{0}(\Omega)$ and $u_{1} \in H^{1}_{0}(\Omega)$. The main purpose of the paper is to study the dynamical behavior of the solution of dissipative system (\ref{1}) damped by time-dependent frictions. It is clear that the nonlinearity $f$ and the coefficient $h(t)$ play significant roles in the analysis. In particular, the set of equilibria associated to system (\ref{1}) depends on the assumptions on the nonlinearity $f$. And whether $h(t)$ is asymptotically vanishing or becoming larger too rapidly as time goes to infinity, convergence to some equilibrium for the solution of system (\ref{1}) may fail. 

Let us first begin the introduction with a short survey in the literature. 
The dynamical behavior of solution to system (\ref{1}) and its ODE version
\[
	\ddot{x}(t)+h(t)\dot{x}(t) +\nabla \Phi(x(t))=0, \quad t\geq 0
\]
with a potential $\Phi \in W^{2, \infty}_{\textrm{loc}}(\mathbb{R}^{d}, \mathbb{R})$, which models the heavy ball with friction, have been studied by many authors under various assumptions on the damping and potential terms (cf., e.g., \cite{PS, CEG, AGP, Dao, HJ, Gao2} and references therein). Recently, assuming the nonlinearity $f$ is monotone and conservative, which is equivalent to the convexity of the potential, the authors in \cite{CF, May} proved that the solution converges weakly to an equilibrium if the damping coefficient behaves like $\frac{C}{t^{\alpha}}$ for some $C>0$ and $\alpha\in(0, 1)$. However, in the case of nonconvex force potential, the question of convergence to an equilibrium for the solution of the system (\ref{1}) is left open. The new estimates in this paper allow to solve this question and to propose necessary and sufficient conditions on the damping coefficient for convergence, see main theorem and remarks below.

Before giving the detailed statement of our main theorem, we make assumptions on the nonlinearity $f$ and the damping coefficient $h$:
\begin{itemize}
	\item[(I-1)] $f\in W^{1, \infty}_{\textrm{loc}}(\mathbb{R})$ satisfies
		\begin{eqnarray}
			\liminf\limits_{|s|\rightarrow+\infty}\frac{f(s)}{s}\geq-\mu_{1}, \label{2}
		\end{eqnarray}	
		where the constant $\mu_{1}<\mu_{0}$, and $\mu_{0}$ is the first eigenvalue of $-\Delta$ in $\Omega$ with zero Dirichlet boundary condition; and that 
		\begin{eqnarray}
			|f^{\prime}(s)| \leq C(1+|s|^{p}), \quad s\in \mathbb{R}, \label{3}
		\end{eqnarray}
		where $C>0$ and $p\geq 0$, with $(d-2)p<2$, are constants. 
	\item[(I-2)] Put 
			\[
				G(v):= -\Delta v +f(v), \quad E_{0}(v):=\frac{1}{2}\|\nabla v \|^{2}_{L^2} +\int_{\Omega}F(v)dx,
			\]
			where $F(s):=\int_{0}^{s}f(t)dt$, and
			\[
				S:=\{ \psi \in H^{2}(\Omega)\cap H^{1}_{0}(\Omega): G(\psi)=0\}
			\]
			(the set of equilibria). There exists a number $\theta \in (0, \frac{1}{2}]$ such that for each $\psi \in S$,
			\begin{eqnarray}
				\|G(v)\|_{H^{-1}}\geq c_{\psi}|E_{0}(v)-E_{0}(\psi)|^{1-\theta}, \label{4}
			\end{eqnarray}
			whenever $v\in H^{1}_{0}(\Omega)$, $\|v-\psi\|_{H^{1}}\leq \sigma_{\psi}$; here, $c_{\psi}$ and $\sigma_{\psi}$ are constants depending on $\psi$.
	\item[(I-3)] The damping coefficient $h(t)\in W_{\mathrm{loc}}^{1, \infty}(\mathbb{R}^{+})$ is a nonnegative function, and there exist constants $c, C>0$, and 
		\begin{eqnarray}
			\alpha \in [0, \theta(1-\theta)^{-1}) \label{5}
		\end{eqnarray}
		such that
		\begin{eqnarray}
			\frac{c}{(t+1)^{\alpha}}\leq h(t) \leq \frac{C}{(t+1)^{\alpha}}, \quad \forall t\geq 0, \label{6}
		\end{eqnarray}
		or
		\begin{eqnarray}
			ct^{\alpha} \leq h(t) \leq Ct^{\alpha}, \quad \forall t\geq 0. \label{7}
		\end{eqnarray} 
\end{itemize}
 In \cite{J1}, the authors proved the convergence result under the conditions (I-1), (I-2), (I-3) without (\ref{7}), and the following assumption (I-4).
\begin{itemize}
	\item[(I-4)] For any $a>0$,
		\begin{eqnarray*}
			\inf\limits_{t>0}\int_{t}^{t+a}h(s)ds>0. 
		\end{eqnarray*}
\end{itemize}
Condition (I-4) is a technical assumption, which is only used to show decay to zero of $u_t$ in $L^2$ and then the precompactness of the trajectories of system (\ref{1}). However, condition (I-4) implies that $h$ does not tend to $0$ at infinity. Stimulated by all the work above, the major contribution of this paper is to present an effective method to prove the convergence to equilibrium of solutions of system (\ref{1}) without the assumption (I-4). More precisely, we will prove the main theorem as follows.
\begin{theorem}\label{T1}
	Assume Conditions (I-1), (I-2) and (I-3). Let $u \in W^{1, \infty}_{\textrm{loc}}(\mathbb{R}^{+}, H^{1}_{0})\cap W^{2, \infty}_{\textrm{loc}}(\mathbb{R}^{+}, L^{2})$ be a solution of (\ref{1}). Then $\{(0, \psi): \psi\in S\}$ is the attracting set for system (\ref{1}), that is, 
	\[
		\lim_{t\rightarrow+\infty}(\| u_{t}(t, \cdot)\|_{L^2}+ \| u(t, \cdot)-\psi\|_{H^1})=0.
	\]
Moreover, there exist positive constants $c$, and $C$ such that
	\begin{eqnarray*} 
	\|u-\psi \|_{L^2} \leq
	\left
    	\{
        		\begin{array}{ll}
             		c(1+t)^{-\lambda}, & \theta\in(0, \frac{1}{2}),  \\
                           C\exp(-ct^{1-\alpha}),  & \theta = \frac{1}{2},
        \end{array}
\right.
\end{eqnarray*}
where $\lambda\in (0, \frac{\theta-(1-\theta)\alpha}{1-2\theta})$.	  
\end{theorem}

\begin{remark}
It follows from \cite{HJBook} that (I-2) holds true if $f$ is analytic. For example, the analytic function $f(u)=b\sin u$, $b$ some positive constant, with Lojasiewicz exponent $\theta=\frac{1}{2}$. Then equation (\ref{1}) is a damped sine-Gordon equation. 

As we know in \cite{Chill}, (I-2) is also suitable for some non-analytic functions, for instance, 
\[
	f(u) =au+|u|^{p-1}u, \quad p>1.
\]
And if $a>-\mu_0$, then (I-2) is satisfied with the Lojasiewicz exponent $\theta=\frac{1}{2}$; otherwise, (I-2) holds with $\theta=\frac{1}{p+1}$. Thus, equation (\ref{1}) becomes a damped nonlinear Klein-Gordon equation.
\end{remark}

\begin{remark}
The condition (I-3) on the damping coefficient is optimal, which prevent the damping term from being either too small, or too large as $t\rightarrow +\infty$.

From Theorem \ref{T1}, we obtain the convergence results for equation (\ref{1}) with a small damping coefficient $h(t)\sim \frac{1}{(t+1)^{\alpha}}$, asymptotically vanishing, or a large damping coefficient $h(t)\sim t^{\alpha}$, $\alpha\in[0, \theta(1-\theta)^{-1})$. Here, the notation $\sim$ means that the coefficient grows like a polynomial function.

As for $h(t)= (t+1)^\alpha$, $|\alpha|>1$, Theorem \ref{T1} do not apply, and the solution $u(t, x)$ may oscillate or approach to some functions (not an equilibrium) as time goes to infinity.

Indeed, if $\alpha>1$, there exist oscillating solution that do not approach zero as $t\rightarrow \infty$. 
For example, we consider the problem
\begin{eqnarray*}
\left
    \{
        \begin{array}{ll}
            u_{tt}-\Delta u+ \frac{1}{(t+1)^{\alpha}}u_t +b u=0 &(t, x)\in\mathbb{R}^{+} \times \Omega,  \\
                                u = 0  & (t, x)\in \mathbb{R}^{+} \times \partial\Omega,\\
                  (u(0, x), u_t(0, x))= (u_{0}(x), u_{1}(x)) & x\in \Omega,
        \end{array}
\right.
\end{eqnarray*}
where $\mu$ is an eigenvalue of $-\Delta +b $, having the corresponding eigenfunction $\psi(x)$. Taking $u(t, x)=\omega(t)\psi(x)$, we have
\[
	\omega_{tt} + \frac{1}{(t+1)^{\alpha}}\omega_{t}+\mu\omega=0.
\]

When  $\alpha <-1$, solutions again do not in general approach zero, though their behavior is quite different from the case $\alpha<-1$. Note that an interesting solution
\[
	u(t, x)=(1-\frac{\mu(t+1)^{1+\alpha}}{1+\alpha})\psi(x)
\]
solves
\begin{eqnarray*}
\left
    \{
        \begin{array}{ll}
            u_{tt}-\Delta u+ h(t) u_t +b u=0 &(t, x)\in\mathbb{R}^{+} \times \Omega,  \\
                                u = 0  & (t, x)\in \mathbb{R}^{+} \times \partial\Omega,
        \end{array}
\right.
\end{eqnarray*}
where the damping coefficient is as follows
\[
	h(t)=t^{-\alpha}-\frac{\mu}{1+\alpha}t-\alpha t^{-1},
\]
and the initial data 
\[
	(u(0, x), u_t(0, x))= (\frac{1+\alpha-\mu}{1+\alpha}\psi(x), -\frac{\mu}{1+\alpha}\psi(x)), \quad x\in \Omega.
\]
But it is easy to see that $u(t, x)$ approaches to $\psi(x)$, which is not an equilibrium, as $t\rightarrow \infty$.
\end{remark}

The plan of this paper is as follows. In Section 2, we make some estimates of solutions. A key lemma in this section is to prove a generalized type of Lojasiewicz-Simon inequality, which is firstly given in the literature. In Section 3, we prove our main results. And we will give some numerical analysis to illustrate our results in Section 4. Throughout the paper, $c$, $c_1$, $c_2\cdots$, $C$, $C_1$, $C_2\cdots$ denote corresponding constants.

\section{Preliminary}
By the semigroup theory (e.g., \cite{Pazy}, section 6.1) and regard $h(t)u_{t}+f(u)$ as a perturbation term), we know that
under Condition (\ref{3}) and for $u_{0} \in H^{2}(\Omega)\cap H^{1}_{0}(\Omega)$ and $u_{1} \in H^{1}_{0}(\Omega)$, Problem (\ref{1}) has a unique solution
\[
	u(t, x) \in W^{1, \infty}_{\textrm{loc}}(\mathbb{R}^{+}, H^{1}_{0}(\Omega))\cap W^{2, \infty}_{\textrm{loc}}(\mathbb{R}^{+}, L^{2}(\Omega)).
\]
The solution energy is defined by
\[
	E_{u}(t):=\frac{1}{2}(\|u_{t}\|^{2}_{L^2}+\|\nabla u\|^{2}_{L^2})+\int_{\Omega}F(u)dx.
\]
Notice that $E_{u}(\cdot)$ is non-increasing by $E_{u}^{\prime}(t)=-h(t)\|u_{t}\|^{2}_{L^2}\leq 0$.
\begin{lemma} \label{L1}
Let $u\in W^{1, \infty}_{\mathrm{loc}}(\mathbb{R}^{+}, H^{1}_{0})\bigcap W^{2, \infty}_{\mathrm{loc}}(\mathbb{R}^{+}, L^{2})$
be a solution to (\ref{1}). There is a positive constant $c$ depending only on the norm of initial data in the energy space $H^1_{0}(\Omega)\times L^2(\Omega)$, such that
\begin{eqnarray}
	\|u_t\|_{L^2}+\|u\|_{H^1}+\|f(u)\|_{L^2} \leq c \quad \textrm{for $t\geq 0$}. \label{L1-1}
\end{eqnarray}
And the functional $E_{0}(u)$ has at least a minimizer $v\in H^{1}(\Omega)$.
\end{lemma}
\begin{proof}
For every $\delta>0$, it follows from (\ref{2}) that there exists $M(\delta)\gg 1$ such that for $|s|>M(\delta)$ 
\[
	\frac{f(s)}{s}\geq -(\mu_{1}+\delta),
\]
and so
\[
	F(s)\geq -\frac{\mu_{1}+\delta}{2}s^2.
\]
Then we have
\begin{eqnarray*}
\begin{aligned}	
	\int_{\Omega}F(u)dx =&\int_{|u|>M(\delta)}F(u)dx+\int_{|u|\leq M(\delta)}F(u)dx \\
			                \geq& -\frac{\mu_{1}+\delta}{2}\int_{\Omega}|u|^2 dx+|\Omega|\sup_{|u|\leq M(\delta)}|F(s)|,
\end{aligned}			   
\end{eqnarray*}
which implies that
\begin{eqnarray}
\begin{aligned}	 \label{L1-2}
	E_{u}(t)\geq \frac{1}{2}\|u_{t}\|^{2}_{L^2}+ \frac{1}{2}\big(1-\frac{\mu_{1}+\delta}{\mu_{0}} \big)\|\nabla u\|^{2}_{L^2}+|\Omega|\sup_{|u|\leq M(\delta)}|F(s)|.
\end{aligned}			   
\end{eqnarray}
Since $E_{u}(\cdot)$ is non-increasing, then by using the Poincar\'e inequality, it implies from (\ref{L1-2}) that
\[
	\|u_{t}\|^{2}_{L^2}+ \| u\|^{2}_{H^1}\leq c_1,
\]
where $c_1$ depends on the norm of initial data in the energy space. Accordingly, $\|f(u)\|_{L^2}$ is also bounded by (I-1). Then (\ref{L1-1}) is proved.

And also we know from (\ref{L1-2}) that $E_0(u)$ is bounded below. Thus, there exists a minimizing sequence $u_n\in H^{1}(\Omega)$ such that $E_{0}(u_n)=\inf_{u\in H^{1}} E_{0}(u)$. By (\ref{3}) H\"older inequality and Sobolev inequalities, we have
\begin{eqnarray*}
\begin{aligned}
	\|f(u_n)-f(v)\|^2_{L^2}&\leq c_2\int_{\Omega}(|u_n|^{2p}+|v|^{2p}+1)|u_n-v|^2 dx \\
					&\leq c_3 \|(|u_n|^{2p}+|v|^{2p}+1)\|_{L^{\frac{3p}{1-q}}}\||u_n-v|^2\|_{L^{\frac{6}{1+2q}}} \\
					&\leq c_4( \|u_n\|_{H^{1-q}}^{2p}+\|v\|_{H^{1-q}}^{2p}+1)\|u_n-v\|^2_{H^{1-q}}
\end{aligned}	
\end{eqnarray*}
with $q=\frac{2-p}{2(1+p)}\in(0, 1)$. 
From (\ref{L1-1}) and the Aubin's compactness theorem, we know that $u_n$ is relatively compact in $L^{\infty}([0, T], H^{1-q}(\Omega))$. Then there is a subsequence, denoted still by $u_n$, such that $u_n \rightarrow v $ in $H^{1-q}$, and $v\in H^{1}$. 
It implies
\begin{eqnarray*}
\begin{aligned}
	|\int_{\Omega}F(u_n)-F(v)dx| &= |\int_{\Omega}[\int_{0}^{1}f(su_n)u_n-f(sv)v ds]dx|\\
						&=|\int_{\Omega}[\int_{0}^{1}f(su_n)u_n-f(su_n)v+f(su_n)v-f(sv)v ds]dx|\\
						&\leq c_5\{\|u_n -v\|_{L^2}+ \|f(su_n) -f(sv)\|_{L^2} \}\rightarrow 0
\end{aligned}
\end{eqnarray*}
as $n$ goes to infinity. Because $\|u\|_{H^1}$ is weakly lower semicontinuous, we obtain $E_{0}(v)=\inf_{u\in H^{1}} E_{0}(u)$, that is, $v$ is a minimizer of the functional $E_{0}(u)$.
\qed
\end{proof}

\begin{remark}
From this Lemma, we know the following static problem associated to system (\ref{1})
\begin{eqnarray*}
\left
    \{
        \begin{array}{ll}
            -\Delta \psi+f(\psi)=0 & x\in  \Omega,  \\
                                \psi= 0  & x\in \partial\Omega
        \end{array} 
\right.
\end{eqnarray*}
admits at least a classical solution, which means that the set of equilibria is nonempty.

\end{remark}

The following inequality is a generalized type of Simon-Lojasiewicz inequality. 
\begin{lemma} \label{L2}
Assume that $u\in W^{1, \infty}_{\mathrm{loc}}(\mathbb{R}^{+}, H^{1}_{0})\bigcap W^{2, \infty}_{\mathrm{loc}}(\mathbb{R}^{+}, L^{2})$
is a solution to (\ref{1}) and $\psi\in S$. There exist constants $c>0$, $T>0$, $\theta\in(0, \frac{1}{2}]$ and $\beta>0$ depending on $\psi$ such that for
	\begin{eqnarray*}
		\|G(u)\|_{H^{-1}}\geq c|E_{0}(u)-E_{0}(\psi)|^{1-\theta}
	\end{eqnarray*}
	provided $\|u-\psi\|_{H^{1-q}(\Omega)}<\beta$, $q=\frac{2-p}{2(1+p)}$. 
\end{lemma}
\begin{proof}
If $\|u-\psi\|_{H^{1}(\Omega)}\geq\beta$, then for $z=u-\psi$, we have
\begin{eqnarray*}
\left
    \{
        \begin{array}{ll}
            -\Delta z=u_{tt}+h(t)u_t +f(u)-f(\psi)& x\in  \Omega,  \\
                       z= 0  & x\in \partial\Omega
        \end{array} 
\right.
\end{eqnarray*}
From the regularity theory for elliptic problem we know
\begin{eqnarray}
	\|z\|_{H^1}\leq c_1\|\Delta z\|_{H^{-1}}, \label{L2-1}
\end{eqnarray}
where $c_1$ is a constant independent of $u$. From (\ref{3}), we have
\begin{eqnarray*}
\begin{aligned}
	\|f(u)-f(\psi)\|_{L^2} \leq c_2\big( \|u\|_{H^{1-q}}^{p}+\|\psi\|_{H^{1-q}}^{p}+1 \big)\|u-\psi\|_{H^{1-q}}.
\end{aligned}	
\end{eqnarray*}
There exists $\tilde{\beta}>0$ depending on $\psi$ such that for any $v$, $\|u-\psi\|_{H^{1-q}}<\tilde{\beta}$,
\begin{eqnarray}
	\|f(u)-f(\psi)\|_{L^2}<\frac{\beta}{2c_1}. \label{L2-2}
\end{eqnarray}
Then it infers from (\ref{L2-1}), (\ref{L2-2}) and Lemma \ref{L1} that
\begin{eqnarray*}
\begin{aligned}
	\|-\Delta u+f(u)\|_{H^{-1}} &= \|-\Delta z+f(u)-f(\psi)\|_{H^{-1}}\\
					      &\geq \|\Delta z\|_{H^{-1}}-\|f(u)-f(\psi)\|_{H^{-1}}\\
					      &\geq \frac{1}{c_1}\|z\|_{H^{1}}-\|f(u)-f(\psi)\|_{H^{-1}}\\
					      &\geq \frac{1}{2c_1}\|z\|_{H^{1}}\\
					     &\geq  c_3|E_{0}(u)-E_{0}(\psi)|^{1-\theta}.
\end{aligned}
\end{eqnarray*}
For the other case $\|v-\psi\|_{H^{1}(\Omega)}<\beta$, it follows from (I-2) that the estimate in this Lemma holds. Thus, our proof is completed. \qed
\end{proof}

\section{Main Result}
In this section, we give the proof of Theorem \ref{T1}.
\begin{proof}
The subsequent proof consists of several steps.
\newline
\textbf{Step 1.} 
From (\ref{L1-1}) and the Aubin's compactness theorem, we know that $u$ is relatively compact in $L^{\infty}([0, T], H^{1-q}(\Omega))$. It follows that there exist a sequence $t_n\rightarrow +\infty$ and $\psi\in H^{1-q}$ such that
\begin{eqnarray}
	\|u(t_n, \cdot)-\psi\|_{H^{1-q
	}} \rightarrow 0, \quad n \rightarrow+\infty.
\end{eqnarray}
Let $\eta \in (0, 1)$, we define
\begin{eqnarray*}
\begin{aligned}
	H(t)=E_{u}(t)-E_{0}(\psi)+ \frac{\eta}{(t+1)^{\alpha}}\langle G(u), u_t \rangle,
\end{aligned}
\end{eqnarray*}
where $\langle\cdot, \cdot\rangle$ is $H^{-1}$ inner product. 
Then we have
\begin{eqnarray*}
\begin{aligned}
	H^{\prime}(t) =&-h(t)\|u_t\|^{2}_{L^2}-\frac{\eta\alpha}{(t+1)^{\alpha+1}}\langle G(u), u_t\rangle+\frac{\eta}{(t+1)^{\alpha}}\langle G^{\prime}(u)u_t, u_t \rangle\\
	&-\frac{\eta}{(t+1)^{\alpha}}\|G(u)\|_{H^{-1}}-\frac{\eta}{(t+1)^{\alpha}}\langle G(u), h(t)u_t \rangle.
\end{aligned}
\end{eqnarray*}
If $h(t)$ satisfies (\ref{6}), then we obtain
\begin{eqnarray*}
\begin{aligned}
	H^{\prime}(t) 
	\leq &-\frac{c}{(t+1)^{\alpha}}\|u_t\|^{2}_{L^2}-\frac{\eta}{(t+1)^{\alpha}}\|G(u)\|_{H^{-1}}+\frac{\eta}{(t+1)^{\alpha}}\langle G^{\prime}(u)u_t, u_t \rangle\\
	&+(\frac{\eta\alpha}{(t+1)^{\alpha+1}}+\frac{\eta C}{(t+1)^{2\alpha}})|\langle G(u), u_t\rangle|.
\end{aligned}
\end{eqnarray*}
And when $h(t)$ satisfies (\ref{7}), then we obtain
\begin{eqnarray*}
\begin{aligned}
	H^{\prime}(t) 
	\leq &-ct^{\alpha}\|u_t\|^{2}_{L^2}-\frac{\eta}{(t+1)^{\alpha}}\|G(u)\|_{H^{-1}}+\frac{\eta}{(t+1)^{\alpha}}\langle G^{\prime}(u)u_t, u_t \rangle\\
	&+(\frac{\eta\alpha}{(t+1)^{\alpha+1}}+\frac{\eta Ct^{\alpha}}{(t+1)^{\alpha}})|\langle G(u), u_t\rangle|\\
	\leq &-ct^{\alpha}\|u_t\|^{2}_{L^2}-\frac{\eta}{(t+1)^{\alpha}}\|G(u)\|_{H^{-1}}+\frac{\eta}{(t+1)^{\alpha}}\langle G^{\prime}(u)u_t, u_t \rangle\\
	&+(\frac{\eta\alpha}{(t+1)^{\alpha+1}}+\eta C)|\langle G(u), u_t\rangle|.
\end{aligned}
\end{eqnarray*}
Since we know
\[
	|\langle G(u), u_t\rangle|\leq \|G(u)\|_{H^{-1}}\|u_t\|_{H^{-1}}\leq \frac{1}{2}\|G(u)\|^{2}_{H^{-1}}+\frac{1}{2}\|u_t\|^{2}_{L^{2}},
\]
or
\[
	|\langle G(u), u_t\rangle|\leq \|G(u)\|_{H^{-1}}\|u_t\|_{H^{-1}}\leq \frac{C_1}{2(1+t)^{\alpha}}\|G(u)\|^{2}_{H^{-1}}+\frac{t^{\alpha}}{2C_2}\|u_t\|^{2}_{L^{2}},
\]
moreover, 
\[
	\|(-\Delta)^{-1}(f^{\prime}(u))u_t\| \leq C_3(1+\|u\|_{H^1})\|u_t\|_{L^2}
\]
by (\ref{3}), and so
\[
	|\langle G^{\prime}(u)u_t, u_t\rangle|\leq C_4\|u_t\|^{2}_{L^2},
\]
then it implies from choosing $\eta$ small enough that
\begin{eqnarray}
\begin{aligned}
	H^{\prime}(t) \leq-\frac{C_5}{(t+1)^{\alpha}}(\|u_t\|^{2}_{L^2}+\|G(u)\|^{2}_{H^{-1}}) \label{T2}
\end{aligned}
\end{eqnarray}
regardless of whether $h(t)$ satisfies (\ref{6}) or (\ref{7}). And (\ref{T2}) implies that $H(t)$ is non-increasing. It follows that $H(t)$ has a finite limit as time goes to $+\infty$. 
\newline
\textbf{Step 2.}
Due to the Lemma \ref{L1}, $u$ is weakly compact in $H^1$. And then we note that 
\begin{eqnarray*}
\begin{aligned}
	&\|G(u(t, \cdot))-G(u(s, \cdot))\|_{H^{-1}}\\
	&= \sup_{\zeta\in Z}\{|\langle G(u(t, \cdot))-G(u(s, \cdot)), \zeta \rangle|\}\\
	&\leq \sup_{\zeta\in Z}\{|\langle \Delta u(t, \cdot)-\Delta u(s, \cdot), \zeta \rangle|\}+\sup_{\zeta\in Z}\{|\langle  f(u(t, \cdot))-f(u(s, \cdot)), \zeta \rangle|\}\\
	&\rightarrow 0, \quad t\rightarrow \infty.
\end{aligned}
\end{eqnarray*}
where $Z=\{\zeta\in H^{1}_{0}: \|\zeta\|_{H^1}=1\}$.
Then we have
\begin{eqnarray} \label{LD2}
\begin{aligned}
	\|G(u)-\int_{t}^{t+1}G(u)ds \|_{H^{-1}} &\leq \int_{t}^{t+1}\|G(u(t, \cdot))-G(u(s, \cdot)) \|_{H^{-1}}ds\\
								&\rightarrow 0, \quad t\rightarrow +\infty.
\end{aligned}
\end{eqnarray}
Because $\|u_t\|_{H^{-1}}$ is uniformly continuous, we deduce that 
\begin{eqnarray}
\|u_t(t+1, \cdot)-u_t(t, \cdot) \|_{H^{-1}}\rightarrow 0, \quad t\rightarrow +\infty. \label{LD3}
\end{eqnarray}
If $h(t)$ satisfies (\ref{6}), we have
\begin{eqnarray}  \label{LD4}
\begin{aligned}
	 \|\int_{t}^{t+1}h(s)u_sds\|_{H^{-1}} &\leq  \|\int_{t}^{t+1}h(s)u_sds\|_{L^{2}}\leq  \int_{t}^{t+1}h(s)\|u_s\|_{L^{2}}ds\\
	 							&\leq (\int_{t}^{t+1}h(s)ds)^{\frac{1}{2}}(\int_{t}^{t+1}h(s)\|u_s\|^{2}_{L^{2}}ds)^{\frac{1}{2}}\\
	 							&\leq C_6(\int_{t}^{t+1}h(s)\|u_s\|^{2}_{L^{2}}ds)^{\frac{1}{2}}\\
								&\leq C_6(E(t)-E(t+1))^{\frac{1}{2}}\rightarrow 0, \quad t\rightarrow +\infty.
\end{aligned}
\end{eqnarray}
Let $J(t) = \frac{\|u_t\|^{2}_{H^{-1}}}{(t+1)^{\alpha}}$. We know 
\begin{eqnarray*}
\begin{aligned}
	\frac{d}{dt}J(t)&=-\frac{\alpha}{t+1}J(t)+\frac{2}{(t+1)^{\alpha}}\langle u_t, u_{tt}\rangle\\
					&=-\frac{\alpha}{t+1}J(t)-2h(t)J(t)+\frac{2}{(t+1)^{\alpha}}\langle u_t, -G(u)\rangle\\
					&\leq -(2h(t)-1+\frac{\alpha}{t+1})J(t)+\frac{1}{(t+1)^{\alpha}}\|G(u)\|^{2}_{H^{-1}}\\
					&\leq -(2h(t)-1+\frac{\alpha}{t+1})J(t)-\frac{1}{C_5}H^{\prime}(t)
\end{aligned}
\end{eqnarray*}
by (\ref{T2}). Note that 
\[
	\sup_{t\geq 0}\int_{t}^{t+1}H^{\prime}(s)ds =0,
\]
and for $t>1$, $2h(t)-1+\frac{\alpha}{t+1}>0$ if h(t) satisfies (\ref{7}). Then it follows from a Grownwall type inequality (see Lemma 2.2 in \cite{GP}) that 
\begin{eqnarray*}
	\|u_t\|^{2}_{H^{-1}}\leq C_7(t+1)^{\alpha}e^{-\epsilon t}, \quad t>1.
\end{eqnarray*}
Furthermore, h(t) satisfies (\ref{7}), we have
\begin{eqnarray}  \label{LD5}
\begin{aligned}
	 \|\int_{t}^{t+1}h(s)u_sds\|_{H^{-1}} &\leq  \|\int_{t}^{t+1}h(s)u_sds\|_{L^{2}}\leq  \int_{t}^{t+1}h(s)\|u_s\|_{L^{2}}ds\\
								&\rightarrow 0, \quad t\rightarrow +\infty.
\end{aligned}
\end{eqnarray}

Now we examine the term $\| G(u) \|_{H^{-1}}$. Since we know by (\ref{1}) that
\begin{eqnarray*}
\begin{aligned}
	G(u) &=G(u)-\int_{t}^{t+1}G(u)ds-\int_{t}^{t+1}(u_{ss}+h(s)u_s )ds\\
		&=G(u)-\int_{t}^{t+1}G(u)ds-(u_t(t+1, \cdot)-u_t(t, \cdot))-\int_{t}^{t+1}h(s)u_sds
\end{aligned}
\end{eqnarray*}
then we have
\begin{eqnarray*}
\begin{aligned}
	\| G(u) \|_{H^{-1}} &\leq \|G(u)-\int_{t}^{t+1}G(u)ds \|_{H^{-1}}+\|u_t(t+1, \cdot)-u_t(t, \cdot) \|_{H^{-1}}\\
					&\quad + \|\int_{t}^{t+1}h(s)u_sds\|_{H^{-1}}.
\end{aligned}
\end{eqnarray*}
Thanks to (\ref{LD2}), (\ref{LD3}), (\ref{LD4}) and (\ref{LD5}), we have 
\[
	\lim_{t\rightarrow \infty}\| G(u) \|_{H^{-1}} =0
\]
Therefore, we know $\psi \in H^{1}$ and  
\begin{eqnarray*}
	-\Delta \psi +f(\psi) =G(\psi)=\lim_{n \rightarrow+\infty}G(u(t_n, \cdot)) =0,
\end{eqnarray*}	
which means that $\psi$ is an equilibrium.
\newline
\textbf{Step 3.}
For any $0<\delta<\beta$, there exists an integer $N$ such that for any $n\geq N$
\begin{eqnarray}
	[H(t_n)]^{\theta(1-\gamma)}-[H(t)]^{\theta(1-\gamma)} \leq  \frac{\delta}{2}, \quad t>t_n \geq 0, \label{T5}
\end{eqnarray}
where $\beta$ is the constant in Lemma \ref{L2}.
Define
\begin{eqnarray*}
	\hat{t}_{n} =\sup\{t>t_n: \|u(s, \cdot)-\psi(\cdot) \|_{H^{1-q}}<\beta, \forall s\in [t_n, t] \}.
\end{eqnarray*}
Due to $2(1-2\theta)\geq 0$ and the uniform boundedness in Lemma \ref{L1}, we deduce from Lemma \ref{L2}  that for $t\in [t_n, \hat{t}_{n})$, $n>N$,
\begin{eqnarray} \label{T5}
\begin{aligned}
	&\|u_{t}\|^{2}_{L^2} +\|G(u)\|^{2}_{H^{-1}}\\
	&\geq \|u_{t}\|^{2}_{L^2}+\frac{1}{2}\|G(u)\|^{2}_{H^{-1}}+\frac{1}{2}|E_{0}(u)-E_{0}(\psi)|^{2(1-\theta)}\\
	&\geq  \frac{\|u_{t}\|^{2(1-2\theta)}_{L^2}}{C_8}\|u_{t}\|^{2}_{L^2}+\frac{\|G(u)\|_{H^{-1}}^{2(1-2\theta)}}{C_9}\|G(u)\|^{2}_{H^{-1}}d+\frac{1}{2}|E_{0}(u)-E_{0}(\psi)|^{2(1-\theta)}\\
	&\geq C_{10}\{\|u_{t}\|^{4(1-\theta)}_{L^2}+\|G(u)\|^{4(1-\theta)}_{H^{-1}}+|E_{0}(u)-E_{0}(\psi)|^{2(1-\theta)}\}\\
	&\geq C_{11}\{\|u_{t}\|^{4(1-\theta)}_{L^2}+\|u_{t}\|^{2(1-\theta)}_{L^2}\|G(u)\|^{2(1-\theta)}_{H^{-1}}+|E_{0}(u)-E_{0}(\psi)|^{2(1-\theta)}\}\\
	&\geq C_{12} H(t)^{2(1-\theta)}.
\end{aligned}
\end{eqnarray}
Then we deduce from (\ref{T2}) that for $t\in [t_n, \hat{t}_{n})$, $n>N$,
\[
	H^{\prime}(t)  \leq -\frac{C_{13}}{(t+1)^{\alpha}}H(t)^{2(1-\theta)}, 
\]
so that
\begin{eqnarray} \label{T6}
	H(t)\leq
	\left
    	\{
        		\begin{array}{ll}
             		C_{14}(1+t)^{-\frac{1-\alpha}{1-2\theta}}, & \theta\in(0, \frac{1}{2}),  \\
                           C_{15}\exp(-C_{11}t^{1-\alpha})  & \theta = \frac{1}{2}. 
        \end{array}
\right.
\end{eqnarray}
Taking $\gamma\in(0, 1)$, from (\ref{T2}) and (\ref{T5}) we also have
\begin{eqnarray*}
\begin{aligned}
	-\frac{\mathrm{d}}{\mathrm{d}t}[H(t)^{\theta(1-\gamma)}]=&-\theta(1-\gamma) H^{\prime}(t)[H(t)]^{\theta(1-\gamma)-1}\\ \label{T3}
											\leq &-C_{16}\theta(1-\gamma)(\|u_t\|_{L^2}+\|G(u)\|_{H^{-1}})^{2}H(t)^{\theta(1-\gamma)-1}\\
											\leq & -\frac{C_{17}}{(t+1)^{\alpha}}H(t)^{-\theta\gamma}(\|u_t\|_{L^2}+\|G(u)\|_{H^{-1}})
\end{aligned}
\end{eqnarray*}
for $t\in [t_n, \hat{t}_{n})$, $n>N$. Note that $(1-\alpha)(1-2\theta)^{-1}\theta\gamma>\alpha$. By (\ref{T6}), we obtain
\[
	\frac{C_{17}}{(t+1)^{\alpha}}H(t)^{-\theta\gamma}\rightarrow +\infty, \quad t\rightarrow +\infty.
\]
Therefore,
\begin{eqnarray}
\begin{aligned}
	\|u_t\|_{L^2}\leq -C_{18}\frac{\mathrm{d}}{\mathrm{d}t}[H(t)^{\theta(1-\gamma)}] \label{T8}
\end{aligned}
\end{eqnarray}
for $t\in [t_n, \hat{t}_{n})$, $n$ large enough. Then we see that
\begin{eqnarray*}
\begin{aligned}
	\sup_{t\in [t_n, \hat{t}_{n})} \|u(t, \cdot)-\psi\|_{L^2}\leq  \|u(t_n, \cdot)-\psi\|_{L^2} + C_{19}|\int_{t_n}^{\hat{t}_n}\frac{\mathrm{d}}{\mathrm{d}t}[H(t)^{\theta(1-\gamma)}]dt|,
\end{aligned}
\end{eqnarray*}
which implies that $\hat{t}_{n}=+\infty$ when $n$ is large enough. Then we assert that $u(t, \cdot)$ converges to $\psi$ in $L^{2}$ as $t$ goes to $+\infty$.
Since $u$ is relatively compact in $H^{1-q}(\Omega)$ (see in Lemma {\ref{L1}}), then we have
\begin{eqnarray}
	\|u(t, \cdot)-\psi\|_{H^{1-q}} \rightarrow 0, \quad t\rightarrow+\infty. \label{T7}
\end{eqnarray}	
From (\ref{T7}), we can deduce 
\begin{eqnarray*}
	|\int_{\Omega}F(u)-F(\psi)dx|\rightarrow 0, \quad t\rightarrow+\infty. 
\end{eqnarray*}
Due to
\[
	\lim_{t\rightarrow\infty}H(t)=\lim_{t\rightarrow\infty}\{\| u_{t}(t, \cdot)\|_{L^2}+E_{0}(v)-E_{0}(\psi)\}=0,
\]
then we know $\{(0, \psi): \psi\in S\}$ is the attracting set for system (\ref{1}), that is, 
\[
	\lim_{t\rightarrow+\infty}(\| u_{t}(t, \cdot)\|_{L^2}+ \| u(t, \cdot)-\psi\|_{H^1})=0.
\]	
\textbf{Step 4.}
From (\ref{T8}) and (\ref{T7}), it infers that
\[
	\|v(t, \cdot)-\psi\|_{L^{2}} \leq \int_{t}^{\infty}\|v_{t}\|_{L^2}ds \leq C_{20}[H(t)]^{\theta(1-\gamma)},
\]
which implies the required speed estimates from (\ref{T6}). \qed
\end{proof}

\section{Numerical simulation}
In this section, we present several examples to illustrate the evolution of the solutions to the system (\ref{1}). We will show how the role of damping coefficient $h(t)$ affects the dynamical behaviors of the solution. We consider the following one-dimensional equations
\begin{eqnarray} \label{Example1}
\left
    \{
        \begin{array}{ll}
            u_{tt} + h(t) u_{t}  - u_{xx}+ G(u) = 0 &(t, x)\in (0, T)\times(-L, L) \\
            u(t, -L) = u(t, L) = 0 &t\in (0, T) \\
            (u(0,x), u_{t}(0,x) = (f(x), g(x))  &x \in (-L, L)
        \end{array}
\right.
\end{eqnarray}
with the initial condition
\begin{eqnarray*}
\begin{aligned}
	f(x) = 0, \quad g(x) = \frac{4\sqrt{1-c^{2}}}{ \cosh(x\sqrt{1-c^{2}})}.
\end{aligned}
\end{eqnarray*}
Take $h(t)=h_{i}(t)$, $i=0$, $1$, ..., $6$ with
\begin{eqnarray*}
\begin{aligned}
	&h_{0}=0, \quad h_1=1,\quad h_2=\frac{1}{\sqrt{t+1}},\quad h_3=\frac{1}{t+1}, \\
	&h_4=\sqrt{t},\quad h_5=t,\quad h_6=t^{\frac{3}{2}},
\end{aligned}
\end{eqnarray*}
also, we consider two cases of the function G(u): one is
\[
	G(u)=G_{1}(u)=a u + |u|^{p-1}u, \quad p\geq 1,
\] 
and the other is
\[
	G(u)=G_{2}(u) =b \sin u, \quad b>0.
\]
The system \ref{Example1} with the nonlinear term $G_1(u)$ is so-called dissipative Klein-Gordon equation, and with $G_2(u)$ is dissipative sine-Gordon equation.
Here, the Lojasiewicz exponent for sine-Gordon equation is $\frac{1}{2}$. And if $a>-\frac{\pi^2}{L^2}$, which is the first eigenvalue for one-dimensional Laplace operator with Dirichlet boundary condition, then the Lojasiewicz exponent for Klein-Gordon equation is $\frac{1}{2}$. If $a<-\frac{\pi^2}{L^2}$, the Lojasiewicz exponent is $\frac{1}{p+1}$. We will use central differences for both time and space derivatives and the following parameters set:
\begin{center}
\begin{tabular}{l|l}
\hline
Space interval & $L =20$\\
Space discretization & $\Delta x = 0.1$ \\
Time discretization & $\Delta t = 0.05$\\
Amount of time steps & $T= 200$ \\
Velocity of initial wave & $c=0.2$ \\
\hline
\end{tabular}
\end{center}

As for the non-damping case $h(t)=h_0=0$, the numerical results can be seen in figure \ref{fig1} (a) and (b). As can be seen easily, The solutions to these two equations does not decay as time goes to infinity. This is because the systems are conservative, that is, the energy of these two systems $E_{u}(t)$ keep to be some constant $E_{u}(0)$ depended fully on the initial data.

If the systems are damped by time-dependent damping, the numerical results can be seen in figure \ref{fig2} (a) and (b). It is clearly that the damping coefficients affect the dynamical behaviors of the solutions to these two systems. From the figures, we also see that $L^2$-norms of the solutions to these two equations converge to some equilibrium, when $h(t)=h_{i}(t)$, $i=1$, $2$, $4$, $5$. For $h(t)=h_3=\frac{1}{t+1}$, $L^2$-norms of the solutions keep on oscillating as time goes to infinity. And if $h(t)=h_6=t^{\frac{3}{2}}$, $L^2$-norms of the solutions also converge to some constant, but which is not the value of some equilibrium. These numerical results are in accordance with the conclusion given in Theorem \ref{T1}.

How the Lojasiewicz exponent affect the convergence speeds of the solutions to these two equations?
In Theorem \ref{T1}, whether the convergence for the systems is exponential or polynomial depends on the Lojasiewicz exponent $\theta$, but the convergence rate is in terms of the damping coefficient $h(t)$. These can be seen from figure \ref{fig3} (a) and (b).

\begin{remark}
When $h(t)\sim t$, whether convergence results hold in unknown theoretically. From the numerical simulation as in figure \ref{fig2} or \ref{fig3}, we know the solution converges to some equilibrium in $L^2$-norm.
\end{remark}

\section*{Authors contributions}
All the authors contributed equally and significantly in writing this paper. All authors read and approved the final manuscript.

\section*{Declaration of Competing Interest}
The authors declare that they have no competing interests.

\section*{Acknowledgements}
The authors thank the anonymous referees very much for the helpful suggestions.


\bibliography{mybibfile}

\begin{figure}[htbp]
\centering
\subfigure[sine-Gordon equation]{      
\includegraphics[width=9.5cm]{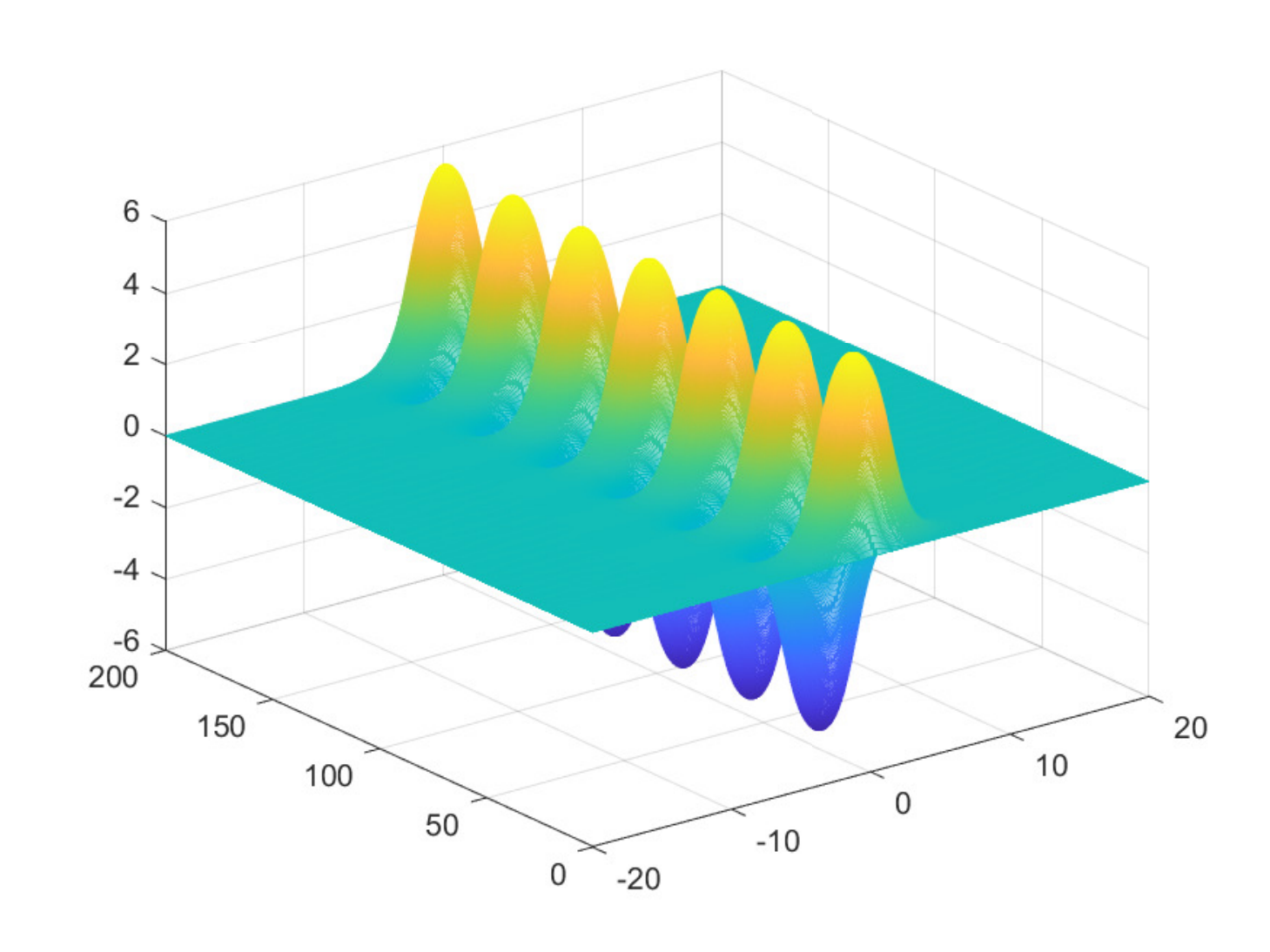}
}
\hspace{0in}
\subfigure[Klein-Gordon equaiton]{
\includegraphics[width=9.5cm]{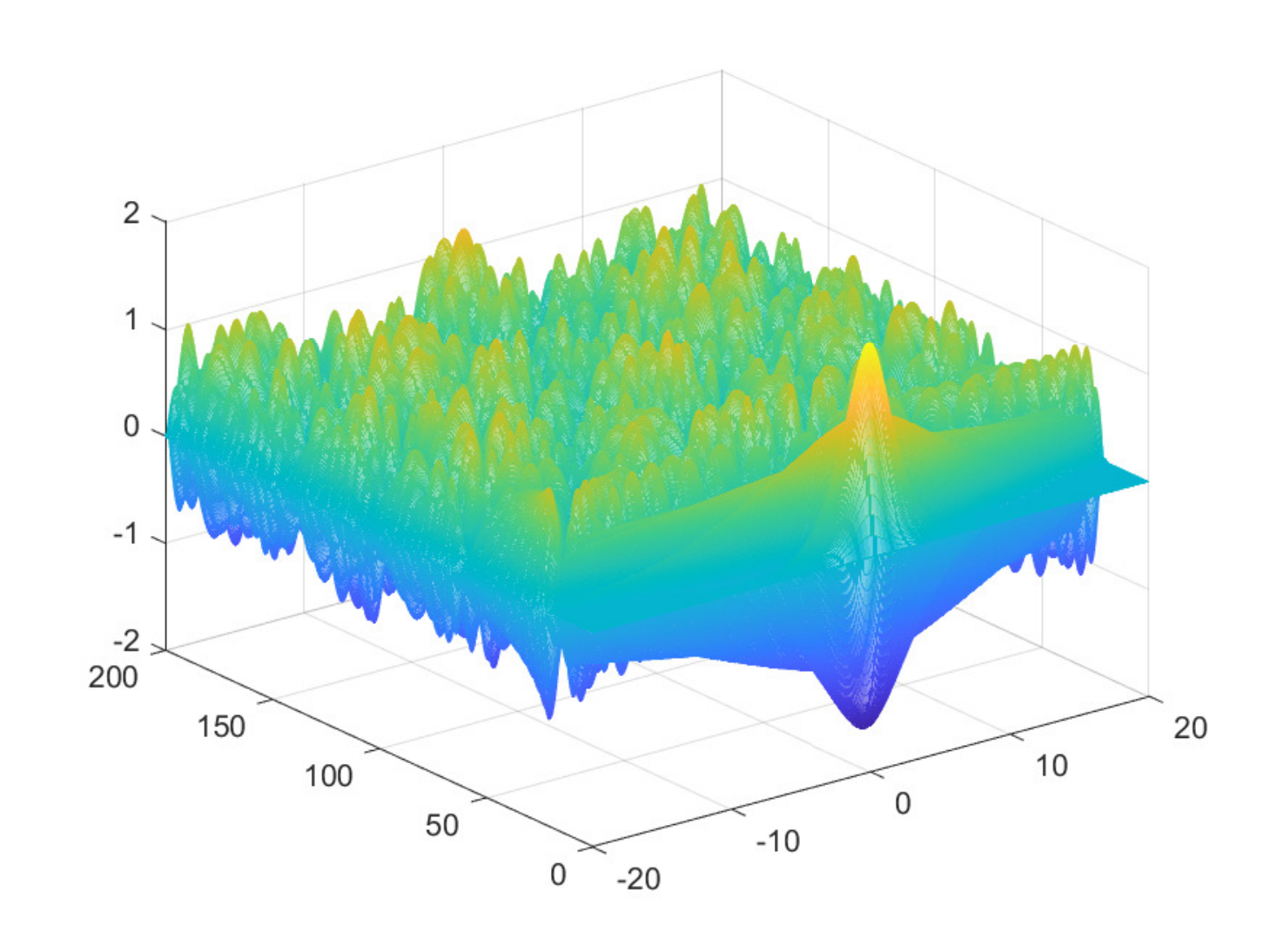}
}
\caption{Dynamical behaviors for two equations without damping. (a) $h(t)=h_0$, $b=1$ The tendency of the solution to sine-Gordon equation. (b) $h(t)=h_0$,  $a=1$, $p=3$ The tendency of the solution to Klein-Gordon equation.}
\label{fig1}
\end{figure}

\begin{figure}[htbp]
\centering
\subfigure[sine-Gordon equation]{      
\includegraphics[width=11cm]{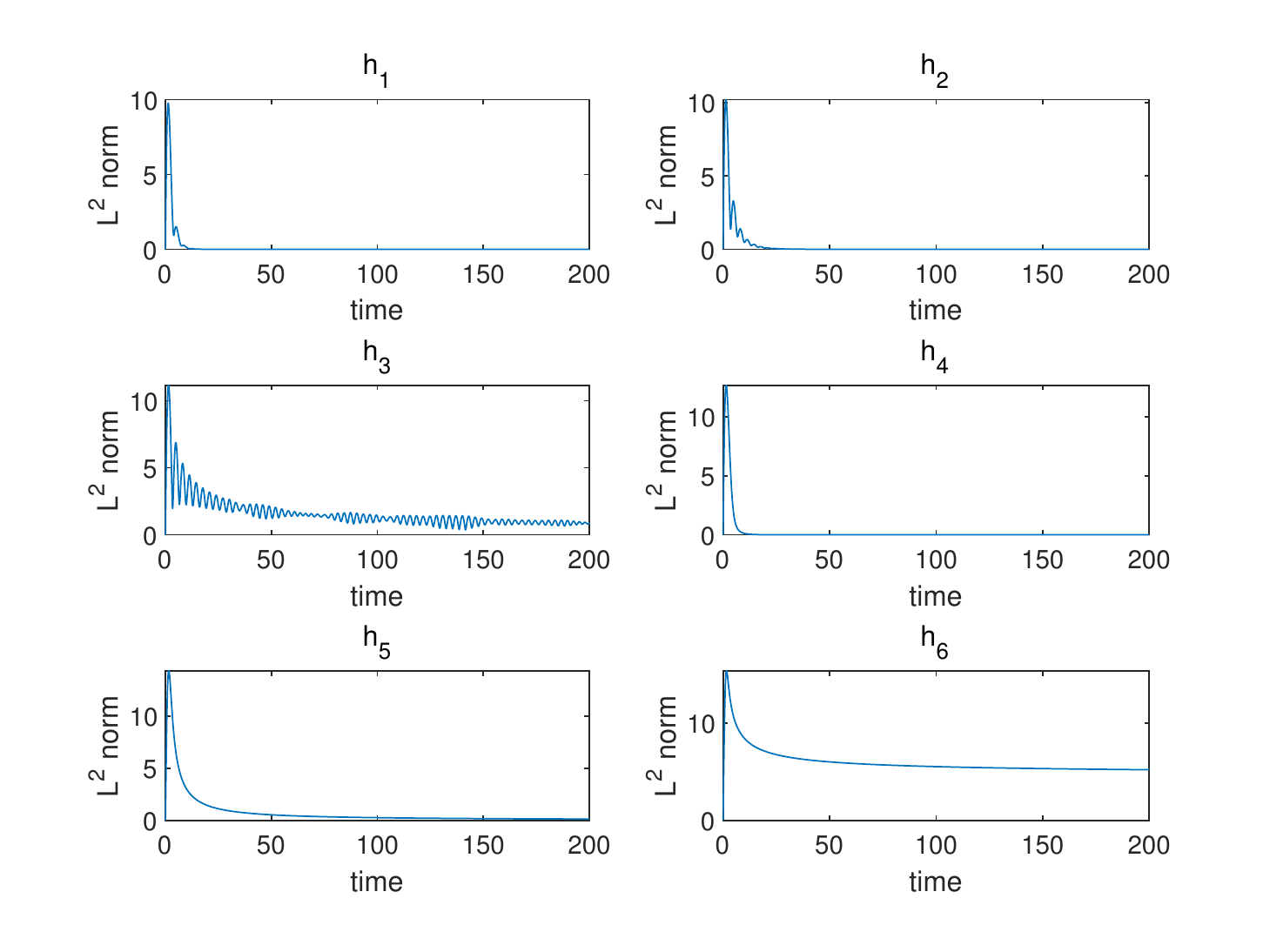}
}
\hspace{0in}
\subfigure[Klein-Gordon equaiton]{
\includegraphics[width=11cm]{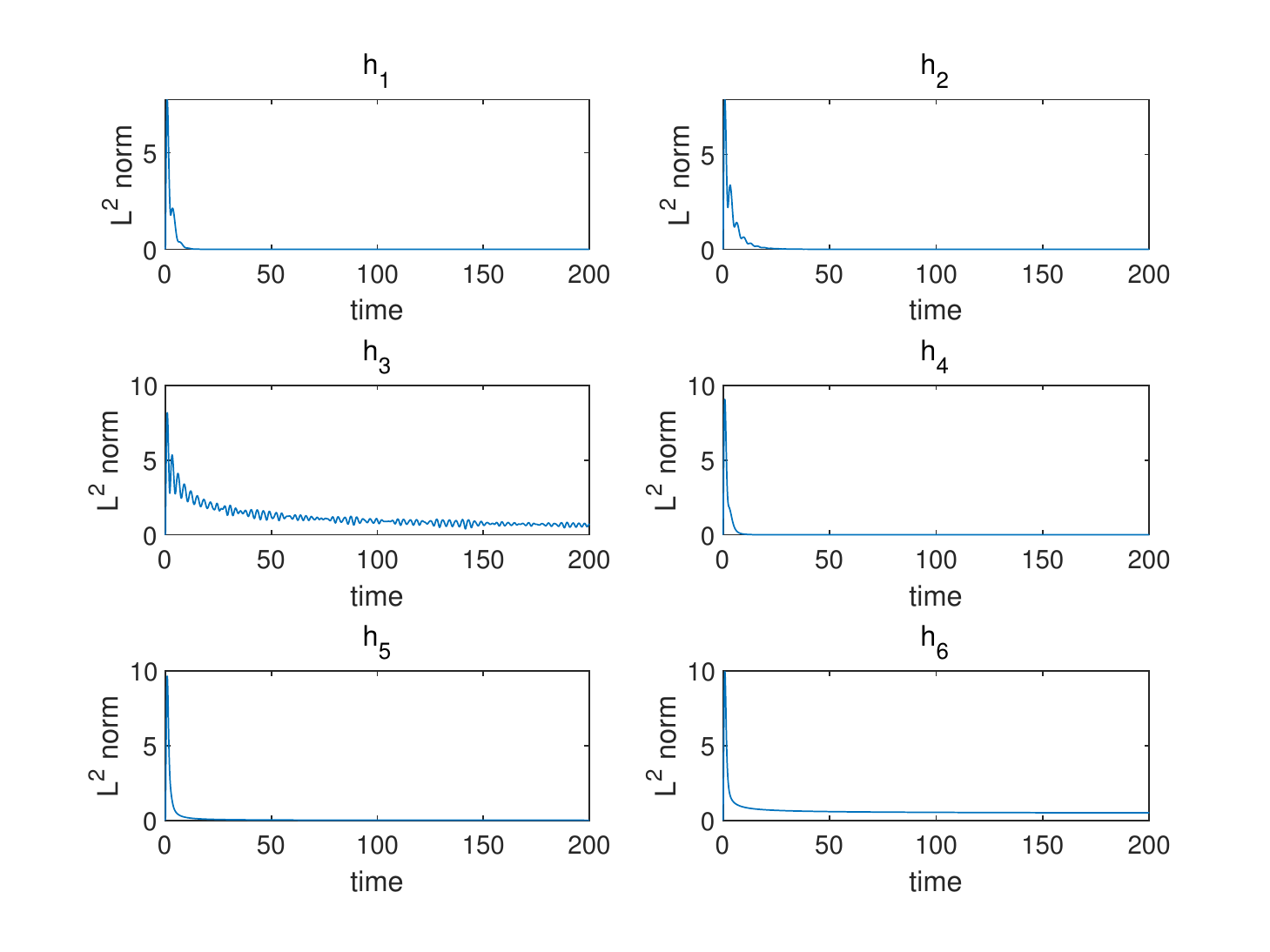}
}
\caption{Dynamical behaviors for two equations with time-dependent damping. (a) $h(t)=h_i$, $b=1$ The tendency of $L^2$-norms of the solution to sine-Gordon equation. (b) $h(t)=h_i$,  $a=1$, $p=3$ The tendency of $L^2$-norms of the solution to Klein-Gordon equation.}
\label{fig2}
\end{figure}

\begin{figure}[htbp]
\centering
\subfigure[$h(t)=h_2$, $p=3$]{      
\includegraphics[width=11cm]{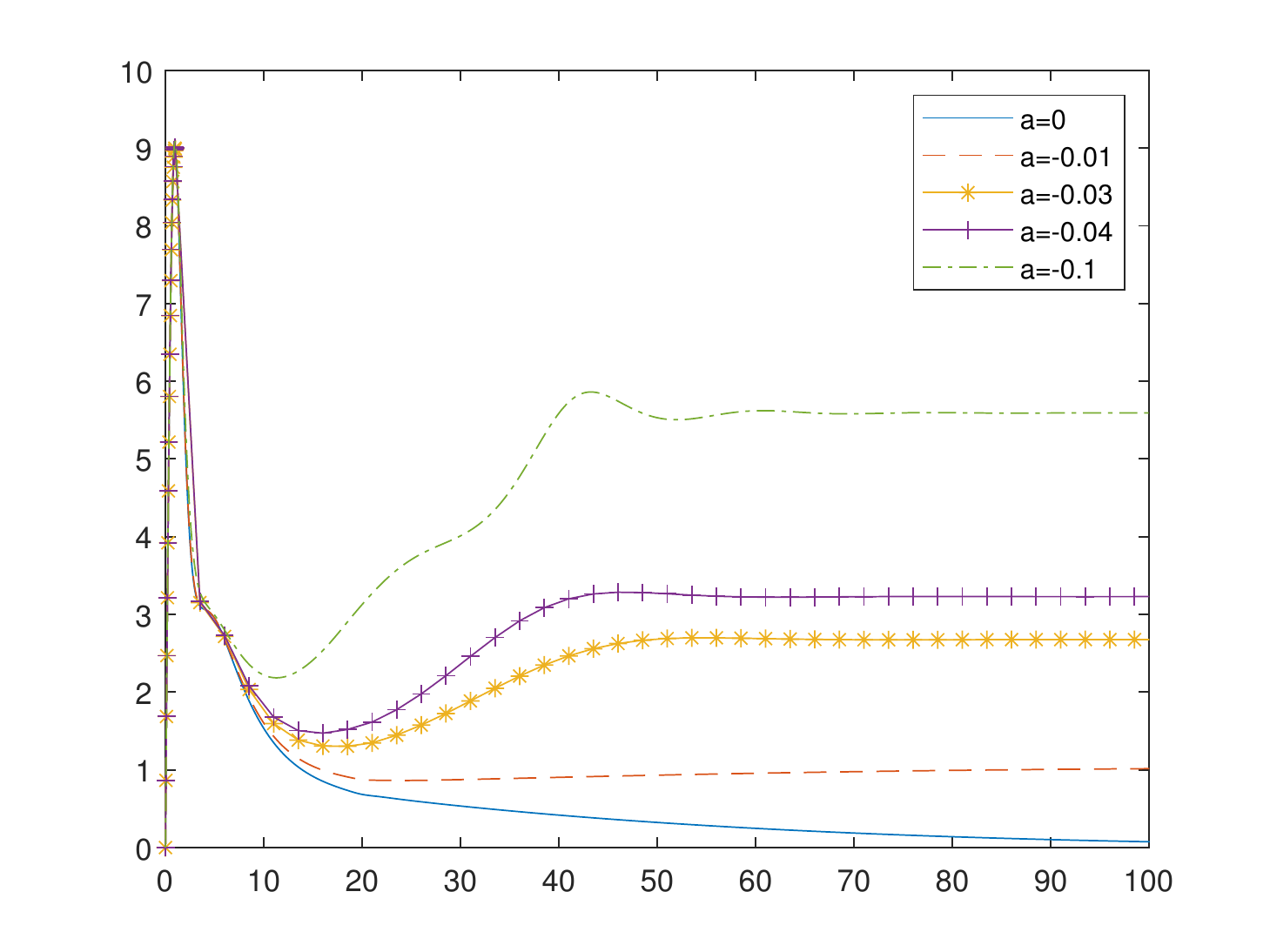}
}
\hspace{0in}
\subfigure[$h(t)=h_i$,  $a=-0.1$, $p=3$]{
\includegraphics[width=11cm]{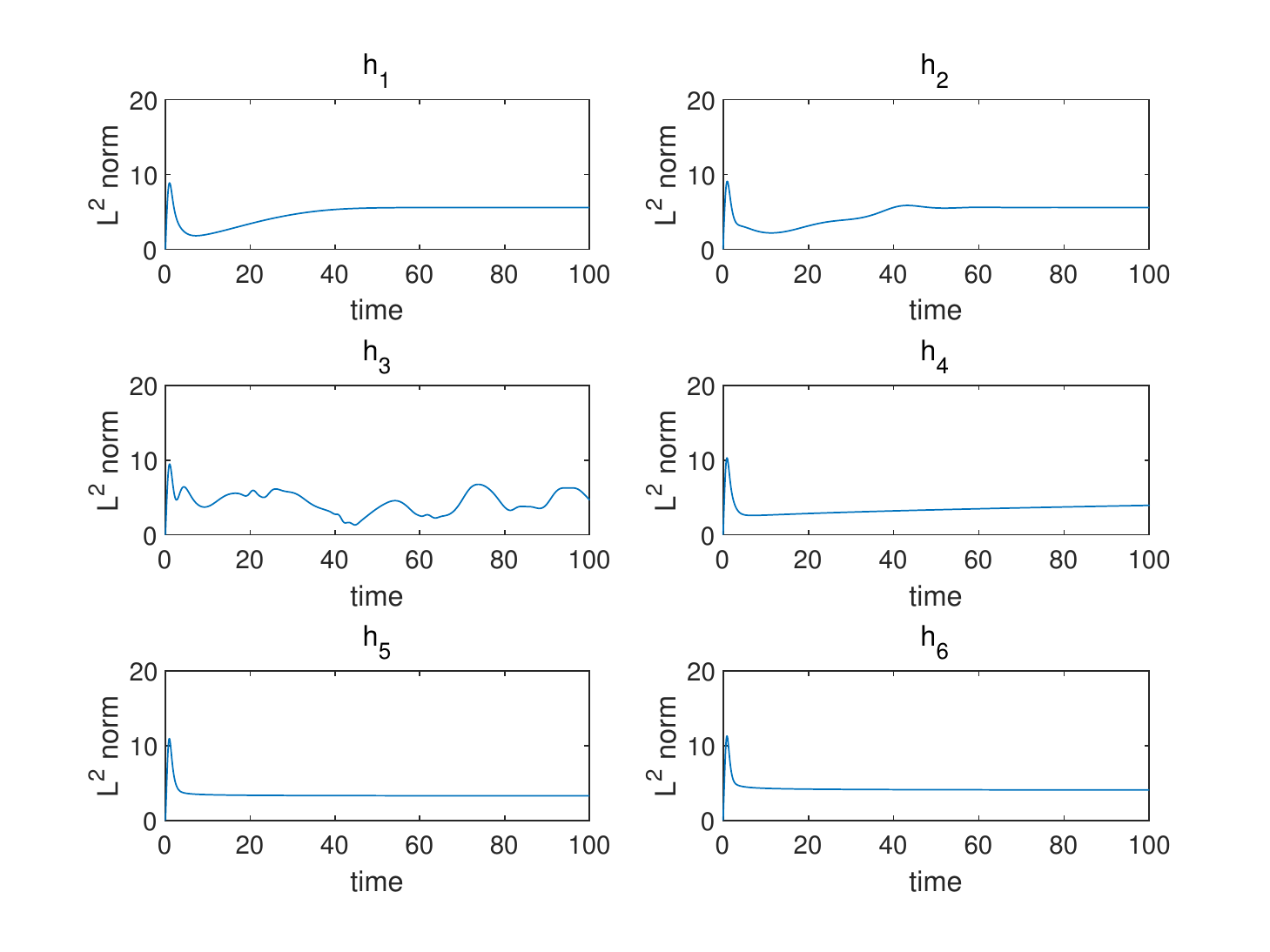}
}
\caption{Convergence rates for dissipative Klein-Gordon equation.}
\label{fig3}
\end{figure}

\end{document}